\documentclass{article}
\usepackage[utf8]{inputenc}

\title{
On the representation of $N$ by the powers of the golden mean and Zeckendorf's theorem}
\author{Edward Zhu}

	\usepackage{hyperref}   
	\hypersetup{colorlinks=true, linkcolor=purple,
		filecolor=magenta, urlcolor=blue,}
\usepackage{xcolor,graphicx,float,amsthm,amsmath,caption,wrapfig}

\usepackage{amssymb}

\usepackage[normalem]{ulem}

\newtheorem{theorem}{Theorem}

\newtheorem{lemma}{Lemma}
\newtheorem{proposition}{Proposition}
\newtheorem{corollary}{Corollary}

\begin{document}

\maketitle

\section{Introduction}

For $\varphi$ the golden ratio,
\[\varphi = \frac{1+\sqrt{5}}{2}, \]
the base-$\varphi$ representation of a positive integer $N$ is an expression of $N$ as a sum of non-consecutive powers of $\varphi$. It is well known that every number has a unique base-$\varphi$ representation \cite{2}. For example,
\[
6 = \varphi^3 + \varphi^1 + \varphi^{-4}.
\]
It's not surprising that we can ``guess" the value of the expansion from just the {\em positive} powers of $\varphi$, because $\varphi^3 + \varphi^1 = 5.854\dots$, which rounds up to $6$. In this paper, we will show that we can also determine the value of the expansion from just the {\em negative} powers of $\varphi$, with the extra information about first non-vanishing positive positive power
  of $\varphi $ in the decomposition. This is surprising, because the negative power is $\varphi^{-4}$, which is $0.145\dots$, which is (we must point out) nowhere near 6.

The base-$\varphi$ representation is related to the Fibonacci sequence and Zeckendorf's theorem.

Let $F_1, F_2, \dots ,$ be the Fibonacci sequence, so $F_0 =0, F_1 =1 , F_{ n +1 } = F_{ n} + F_{n-1}$. The Zeckendorf's theorem states that every positive integer can be written as the unique sum of non-consecutive Fibonacci numbers $F_n, n \geq 2$. For example,
\[6 = F_2 + F_5, \; 7 = F_3 + F_6, \; 8 = F_6, \; 9 = F_2 + F_6, \; 10 = F_3 + F_6\]

The following formulas can be found on page 15 of \cite{1}
\begin{align} \label{Nfn}
 2F_n &= F_{n+1} + F_{n-2}  \\
 3F_n &= F_{n+2} + F_{n-2} \nonumber \\
 4F_n &= F_{n+2} + F_n + F_{n-2} \nonumber \\
 5F_n &= F_{n+3} + F_{n-1} + F_{n-4} \nonumber \\
 6F_n &= F_{n+3} + F_{n+1} + F_{n-4} \nonumber \\
 7F_n &= F_{n+4} + F_{n-4} \nonumber \\
 8F_n &= F_{n+4} + F_{n} + F_{n-4} \nonumber \\
 9F_n &= F_{n+4} + F_{n+1} + F_{n-2} + F_{n-4} \nonumber \\
 10F_n &= F_{n+4} + F_{n+2} + F_{n-2} + F_{n-4} \nonumber \\
 11F_n &= F_{n+4} + F_{n+2} + F_{n} + F_{n-2} + F_{n-4} \nonumber \\
 12F_n &= F_{n+5} + F_{n+1} + F_{n} + F_{n-3} + F_{n-6}\nonumber
\end{align}
These formulas give the Zeckendorf decompositions of the  multiples of Fibonacci numbers, $NF_n$, for $2\leq N \leq 12$, and $n$ is large enough such that all indices on the right hand side are $\geq 2$.

A natural question is for which positive integers $N$, $NF_n$ can be written as
\begin{equation} \label{1.1}
    NF_n = F_{n+i_1} + F_{n+i_2} + F_{n+i_3} + \dots + F_{n+i_m}
\end{equation}
for all $n$, where $i_1 < i_2 < \dots < i_m$ are integers and $i_{k+1} - i_{k} \geq 2$. This problem was answered in \cite{5}, and the answer is related to the base-$\varphi$ representation. Their results state that a base $\varphi$ representation of $N$ leads to an expansion (\ref{1.1}) of $NF_n$, where $\varphi$ is the golden mean.
A base $\varphi$ representation of $N$ can be also expressed as
\begin{equation} \label{1.2}
    N = \sum_{i=R}^{L}d_i \varphi^i
\end{equation}
where $d_i=0$ or $1$ and $R \leq 0 \leq L$ and there are no consecutive powers of $\varphi$, that is, $d_i d_{i+1}=0$. It was shown that if (\ref{1.2}) holds,
then
\begin{equation} \label{1.3}
NF_n = \sum_{i=R}^{L}d_i F_{n+i}
\end{equation}
For all positive integers $N$, we have representation (\ref{1.1}) for $NF_n$.

We adopt the notation used in \cite{3}, for $N$ having the expansion (\ref{1.2}), we write
\begin{equation} \label{bn}
    \beta(N) = d_L d_{L-1} \dots d_1 d_0 . d_{-1} d_{-2} \dots d_{R+1} d_R
\end{equation}
and
\begin{equation} \label{b2}
\beta^{+}(N)= d_L d_{L-1} \dots d_1 d_0 \; \; \;  {\rm and} \; \; \;  \beta^{-}(N)= d_{-1} d_{-2} \dots d_{R+1} d_R
\end{equation}
The properties of base $\varphi$ representations have been studied in \cite{3} and \cite{4}. In particular, the integers with given $d_1 d_0 d_{-1}$ are characterised.

In this work, we study how the positive part $\beta^+(N)$ or negative part
$\beta^-(N)$ in the base $\varphi$ representations of positive integers $N$ determines $N$. Our main result, stated in the beginning, is that $\beta^{-}(N)$, $d_0$ and the parity of the first positive index $i$ with $d_i=1$  determines $N$. To state our result, recall the Lucas numbers $L_n$, defined by $L_0 = 2$, $L_1 = 1$ and the recursive relation $L_n = L_{n-1} + L_{n-2}$.
The Lucas number $L_{-n}$ with negative index $-n$ is defined to be $ L_{-n} = (-1)^n L_n $, so the relation $L_n = L_{n-1} + L_{n-2}$ holds
 for every integer $n$ \cite{7}.
More precisely, our main results are

\begin{theorem} \label{t1}
For a positive integer with representation $\beta(N)$ as in {\rm (\ref{bn})},
let $i$ be the smallest positive index such that $d_i=1$, that is, $d_j =0$ for $0< j<i$, then
\[N = d_L L_{L} + d_{L-1}L_{L-1} + \dots + d_1 L_{1} + d_0 + 1 \; \; \; {\rm for} \; \; i \; {\rm odd} .\]
\[N = d_L L_{L} + d_{L-1} L_{L-1} + \dots + d_1 L_{1} + d_0 \; \; \; {\rm for} \; \; i \; {\rm even}\]
\end{theorem}

Theorem 1 shows that $\beta^{+}(N)$ determines $N$. We have a similar result for $\beta^{-}(N)$, which is

\

\begin{theorem} \label{t2}
For a positive integer with representation $\beta(N)$ as in  {\rm (\ref{bn})}, let $i$ be the smallest positive index such that $i$ such that $d_i=1$.
\[N = d_R (-1)^{-R}L_{-R} + \dots + d_{-1} (-1)^{-1}L_{1} + d_0 \; \; \;{\rm for} \; \; i \; {\rm even}. \]
\[N = d_R (-1)^{-R}L_{-R} + \dots + d_{-1} (-1)^{-1}L_{1} + d_0 -1 \; \; \;{\rm for} \; \; i \; {\rm odd}. \]
\end{theorem}

This result implies that in the expansion (\ref{1.1}), the negative part $\beta^{-}(N)$ together with $d_0$ and the parity of $i$ determines $N$.

One tool we use to prove this theorem is the algebraic conjugation. In Section $2$ we recall the conjugation and prove some basics about base $\varphi$ representations, and we then give a proof of the equivalence for (\ref{1.2}) and (\ref{1.3}). In Section $3$ we prove Theorem $1$ and $2$.

We thank Professor Greg Dresden for his advice and discussions.

\

\section{Zeckendorf representation of $NF_n$ and base $\varphi$ representation of $N$}

We introduce the algebraic conjugate on the set $\mathbb{Q}(\sqrt 5) = \{a+b\sqrt 5 \; | \; a, b \in \mathbb{Q}\}$, which is closed under $+$ and multiplication.
\[ (a+b\sqrt 5)(c+d\sqrt 5) = (ac + 5bd) + (bc + ad) \sqrt 5 \]
The algebraic conjugate of $(a+b\sqrt 5)$ is defined as $\overline{(a+b\sqrt 5)} = (a-b\sqrt 5)$. We have the property
\[ \overline{x+y} = \overline{x} + \overline y, \; \; \;  \overline{xy} = \overline{x} \cdot \overline{y}, \; \; \; \overline{x^n} = \overline{x}^n \]
The algebraic conjugate plays an important role in the proof of our results.
Notice that
\[ \overline{\varphi} = -\varphi^{-1}. \]
Recall that the Fibonacci numbers and the Lucas numbers can be extended to the negative indices, with
\begin{equation} \label{lucas}
    F_n = \frac{\varphi^n - \overline{\varphi}^n}{\sqrt{5}}, \; \;  \; \; L_n = \varphi^n+\overline{\varphi}^n .
\end{equation}

\begin{proposition} The following conditions are equivalent.
\begin{enumerate}
\item $NF_n = F_{n+i_1} + F_{n+i_2} + F_{n+i_3} + \dots + F_{n+i_m}$, for all integers $n$
\item $N = \varphi^{i_1} + \varphi^{i_2} + \dots + \varphi^{i_m}$.
\item $NL_n = L_{n+i_1} + L_{n+i_2} + L_{n+i_3} + \dots + L_{n+i_m}$, for all integers $n$
\end{enumerate}
\end{proposition}

\begin{proof}
We first prove condition 2.~implies conditions 1.~and 3.
Assuming condition 2.,
\[N = \varphi^{i_1} + \varphi^{i_2} + \dots + \varphi^{i_m}\]
multiplying both sides by $\varphi^n$,
\begin{equation} \label{2.1}
N\varphi^n = \varphi^{n+i_1} + \varphi^{n+i_2} + \dots + \varphi^{n+i_m}
\end{equation}
Taking the algebraic conjugate of both sides, we get
\begin{equation} \label{2.2}
N\overline{\varphi^n} = \overline{\varphi^{n+i_1}} + \overline{\varphi^{n+i_2}} + \dots + \overline{\varphi^{n+i_m}}.
\end{equation}
We subtract these two equations, and divide both sides by $\sqrt{5}$, to get
\begin{align*}
N\frac{1}{\sqrt 5}(\varphi^n - \bar{\varphi}^n)
\ = \ \frac{1}{\sqrt 5} \Big(\ \ \ &\varphi^{n+i_1} + \varphi^{n+i_2} + \dots + \varphi^{n+i_m} \\
- &\overline{\varphi}^{n+i_1} - \overline{\varphi}^{n+i_2} - \dots - \overline{\varphi}^{n+i_m}\ \ \Big)
\end{align*}
So
\[NF_n = F_{n+i_1} + F_{n+i_2}  + \dots + F_{n+i_m}\]
we get condition 1~.
(\ref{2.2}) gives us
\[N\overline{\varphi}^n = \overline{\varphi}^{n+i_1} + \overline{\varphi}^{n+i_2} + \dots + \overline{\varphi}^{n+i_m}.\]
Summing the above and (\ref{2.1}), we get
\[N(\varphi^n + \overline{\varphi}^n) = (\varphi^{n+i_1}+\overline{\varphi}^{n+i_1}) + (\varphi^{n+i_2}+\overline{\varphi}^{n+i_2}) + \dots + (\varphi^{n+i_m}+\overline{\varphi}^{n+i_m}) \]
so
\[NL_n = L_{n+i_1} + L_{n+i_2} + \dots + L_{n+i_m} \]
we get condition 3~.
Because for all integers $k$,
\begin{equation}
\label{limFnk}
\lim_{n \to \infty} \frac{F_{n+k}}{F_n} =  \varphi^k
\end{equation}
for all positive integers $k$ \cite{7}.
Assuming condition 1.~,
\[NF_n = F_{n+i_1} + F_{n+i_2} + F_{n+i_3} + \dots + F_{n+i_m}\]
\begin{equation} \label {2}
N = \frac{F_{n+i_1}}{F_n} + \frac{F_{n+i_2}}{F_n} + \frac{F_{n+i_3}}{F_n} + \dots + \frac{F_{n+i_m}}{F_n}.
\end{equation}
Taking the limit $\lim_{n \to \infty}$ on both sides of equation (\ref{2}), apply
(\ref{limFnk}) to get condition 2.
To prove condition 3.~ implies condition 2.~, we use the formula
\[\lim_{n \to \infty}\frac{L_{n+k}}{L_n} = \varphi^k\] for all integers $k$ \cite{7}. The same method for condition 1 implies condition 2 works to prove condition 3 implies condition 2.
\end{proof}
\noindent{\bf Examples.}
In (\ref{Nfn}), we have formulas for $NF_n$ for $2 \leq N \leq 12$. The equivalence of condition 1. and condition 2. in Proposition 1 implies that
\begin{align*}
 2 &= \varphi^{1} + \varphi^{-2} \nonumber \\
 3 &= \varphi^{2} + \varphi^{-2} \nonumber \\
 4 &= \varphi^{2} + \varphi^0 + \varphi^{-2} \nonumber \\
 5 &= \varphi^{3} + \varphi^{-1} + \varphi^{-4} \nonumber \\
 6 &= \varphi^{3} + \varphi^{+1} + \varphi^{-4} \nonumber \\
 7 &= \varphi^{4} + \varphi^{-4} \nonumber \\
 8 &= \varphi^{4} + \varphi^{0} + \varphi^{-4} \nonumber \\
 9 &= \varphi^{4} + \varphi^{1} + \varphi^{-2} + \varphi^{-4} \nonumber \\
 10 &= \varphi^{4} + \varphi^{2} + \varphi^{-2} + \varphi^{-4} \nonumber \\
 11 &= \varphi^{4} + \varphi^{2} + \varphi^{0} + \varphi^{-2} + \varphi^{-4} \nonumber \\
 12 &= \varphi^{5} + \varphi^{1} + \varphi^{0} + \varphi^{-3} + \varphi^{-6}\nonumber  \\
\end{align*}
The equivalence of condition 2. and condition 3. in Proposition 1 implies that
\begin{align}
 2L_n &= L_{n+1} + L_{n-2} \nonumber \\
 3L_n &= L_{n+2} + L_{n-2} \nonumber \\
 4L_n &= L_{n+2} + L_n + L_{n-2} \nonumber \\
 5L_n &= L_{n+3} + L_{n-1} + L_{n-4} \nonumber \\
 6L_n &= L_{n+3} + L_{n+1} + L_{n-4} \nonumber \\
 7L_n &= L_{n+4} + L_{n-4} \nonumber \\
 8L_n &= L_{n+4} + L_{n} + L_{n-4} \nonumber \\
 9L_n &= L_{n+4} + L_{n+1} + L_{n-2} + L_{n-4} \nonumber \\
 10L_n &= L_{n+4} + L_{n+2} + L_{n-2} + L_{n-4} \nonumber \\
 11L_n &= L_{n+4} + L_{n+2} + L_{n} + L_{n-2} + L_{n-4} \nonumber \\
 12L_n &= L_{n+5} + L_{n+1} + L_{n} + L_{n-3} + L_{n-6}\nonumber  \\
\end{align}
\

\section{Proof of Theorem 1 and Theorem 2}
In this section we prove Theorem 1 and Theorem 2. We need the following lemmas in the proof.

\begin{lemma} \label{l1}
    Let $k_1 > k_2 > \dots > k_s$ be integers, and the gaps between every 2 consecutive terms is at least $2$, that is, $k_j-k_{j+1} \geq 2$, then
    \[\varphi^{k_1} + \varphi^{k_2} + \dots + \varphi^{k_s} < \varphi^{k_1+1} \]
\end{lemma}
\begin{proof}
Since $\varphi^2-\varphi-1=0, \; \varphi^{-1}=1-\varphi^{-2}$
\[\varphi = \frac{1}{1-\varphi^{-2}}\]

Now, by our assumptions on $k_1, \dots, k_s$, we have that $k_2 \leq k_1-2$, $k_3 \leq k_2-2 \leq k_1-4$, and in general $k_n \leq k_{n-1}-2 \leq \dots \leq k_1-2(n-1)$, so
\begin{align*}
\varphi^{k_1+1}
    & = \frac{\varphi^{k_1}}{1-\varphi^{-2}}\\
    & = \varphi^{k_1} + \varphi^{k_1-2}+\varphi^{k_1-4}+\dots \\
    & \geq \varphi^{k_1} + \varphi^{k_2} + \varphi^{k_3} + \dots \\
    & > \varphi^{k_1} + \varphi^{k_2} + \varphi^{k_3} + \dots + \varphi^{k_s}
\end{align*}
This proves Lemma 1.
\end{proof}

\

\begin{lemma}
    Let $k_1 > k_2 > \dots > k_s$ be integers, and the gaps between every 2 consecutive terms is at least $2$, that is, $k_j-k_{j+1} \geq 2$.  Then for $k_1$ even, we have
    \[(-1)^{k_1} \varphi^{k_1} + (-1)^{k_2}\varphi^{k_2} + \dots + (-1)^{k_s}\varphi^{k_s} > 0.\]
    For $k_1$ odd, we have
    \[(-1)^{k_1} \varphi^{k_1} + (-1)^{k_2}\varphi^{k_2} + \dots + (-1)^{k_s}\varphi^{k_s} < 0.\]
\end{lemma}

\begin{proof}
By Lemma \ref{l1},
\begin{eqnarray*}
&& \varphi^{k_1} \\
&& > \varphi^{k_2+1} \\
&& > \varphi^{k_2} + \varphi^{k_3} + \varphi^{k_4} + \dots + \varphi^{k_s} \\
&& \geq (-1)^{k_2}\varphi^{k_2}+\dots+(-1)^{k_s}\varphi^{k_s}
\end{eqnarray*}
So
\[\varphi^{k_1} > (-1)^{k_2}\varphi^{k_2} + \dots + (-1)^{k_s}\varphi^{k_s}\]
For $k_1$ odd,
\[ (-1)^{k_1}\varphi^{k_1} = -\varphi^{k_1} < -((-1)^{k_2}\varphi^{k_2}+\dots+(-1)^{k_s}\varphi^{k_s})\]
\[(-1)^{k_1}\varphi^{k_1} + (-1)^{k_2}\varphi^{k_2}+\dots+(-1)^{k_s}\varphi^{k_s} < 0\]
For $k_1$ even,
\begin{eqnarray*}
    && (-1)^{k_1}\varphi^{k_1} + (-1)^{k_2}\varphi^{k_2} + \dots + (-1)^{k_s}\varphi^{k_s} \\
    && \geq (-1)^{k_1}\varphi^{k_1}-(\varphi^{k_2}+\dots+\varphi^{k_s})\\
    && \geq (-1)^{k_1}\varphi^{k_1} - \varphi^{k_2+1}\\
    && > (-1)^{k_1}\varphi^{k_1}-(\varphi^{k_1})\\
    && = \varphi^{k_1} - \varphi^{k_1}\\
    && = 0
\end{eqnarray*}
So for $k_1$ even,
\[(-1)^{k_1}\varphi^{k_1} + (-1)^{k_2}\varphi^{k_2} + \dots + (-1)^{k_s}\varphi^{k_s} > 0.\]
\

This proves Lemma 2.
\end{proof}

\

\begin{proof}[Proof of Theorem 1.]
Let $N$ be a positive integer with representation $\beta(N)$ as in ($\ref{bn}$), so
\[\beta(N) = d_L d_{L-1} \dots d_1 d_0 . d_{-1} d_{-2} \dots d_{R+1} d_R\]
Let $i_1 > \dots > i_m > 0$ be the positive indices with $d_i = 1$. Notice that $i_m$ is the first positive index such that $d_i = 1$. That is, $i_m$ is $i$ in the statement of Theorem 1. Let $-j_1 > \dots > -j_n$ be the negative indices with $d_{-j}=1$. So
\[N = \varphi^{i_1} + \dots + \varphi^{i_m} + d_0 + \varphi^{-j_1} + \dots + \varphi^{-j_n}\]
Adding and subtracting powers of $\bar{\varphi}$, we get
\begin{align*}
    N &= \varphi^{i_1} + \bar \varphi^{i_1} + \dots + \varphi^{i_m} + \bar \varphi^{i_m} \\
    & \;\; -(\bar \varphi^{i_1} + \dots + \bar \varphi^{i_m}) + \varphi^{-j_1} + \varphi^{-j_2} + \dots + \varphi^{-j_k} + d_0\\
\end{align*}
and now we use (\ref{lucas}) to re-write this as
\begin{equation}\label{integer}
N = L_{i_1} + \dots + L_{i_m} - ( (-1)^{-i_1} \varphi^{-i_1} + \dots + (-1)^{-i_m} \varphi^{-i_m} ) + \varphi^{-j_1} + \dots + \varphi^{-j_k} + d_0
\end{equation}

Using Lemma 1, we have
\begin{equation} \label{inq1}
0< \varphi^{ - j_1 } + \varphi^{-j_2 } + \dots + \varphi^{-j_k } < \varphi^{-j_1 +1 } \leq 1
\end{equation}

Using Lemma 2, we have
\begin{eqnarray}
    - \Big(  (-1)^{ - i_1 }\varphi^{- i_1}    +  \dots  + (-1)^{-i_m} \varphi^{- i_m} \Big) > 0  \; \; \; {\rm for} \; i_m \; {\rm odd} \label{odd} \\
    - \Big(  (-1)^{ - i_1 }\varphi^{- i_1}    +  \dots  + (-1)^{-i_m} \varphi^{- i_m} \Big) < 0  \; \; \; {\rm for} \; i_m \; {\rm even} \label{even}
\end{eqnarray}

Using Lemma 1, we have
\[ (-1)^{ - i_1 }\varphi^{- i_1}    +  \dots  + (-1)^{-i_m} \varphi^{- i_m} \leq \varphi^{-i_1} +\dots+\varphi^{-i_m} < 1\]
For $i_m$ odd, by (\ref{odd}),
\[0 < - \Big(  (-1)^{ - i_1 }\varphi^{- i_1}     +  \dots  + (-1)^{-i_m} \varphi^{- i_m} \Big) < 1\]
This with (\ref{inq1}) implies that
\[0 < - \Big(  (-1)^{ - i_1 }\varphi^{- i_1}     +  \dots  + (-1)^{-i_m} \varphi^{- i_m} \Big) + \varphi^{ - j_1 } + \varphi^{-j_2 } + \dots + \varphi^{-j_k } < 2,\]
(\ref{integer}) implies that
\[- \Big(  (-1)^{ - i_1 }\varphi^{- i_1}     +  \dots  + (-1)^{-i_m} \varphi^{- i_m} \Big) + \varphi^{ - j_1 } + \varphi^{-j_2 } + \dots + \varphi^{-j_k }\] is an integer.
So
\[- \Big(  (-1)^{ - i_1 }\varphi^{- i_1}     +  \dots  + (-1)^{-i_m} \varphi^{- i_m} \Big) + \varphi^{ - j_1 } + \varphi^{-j_2 } + \dots + \varphi^{-j_k }=1.\]
By (\ref{integer})
\[N = L_{i_1} + L_{i_2} + \dots + L_{i_m} + d_0+ 1\]
This is equivalent to the first identity of Theorem 1.

For $i_m$ even, by (\ref{even}),
\[-1 < - (  (-1)^{ - i_1 }\varphi^{- i_1}    +  \dots  + (-1)^{-i_m} \varphi^{- i_m} ) < 0\]
This with (\ref{inq1}) implies that
\[-1<- (  (-1)^{ - i_1 }\varphi^{- i_1}     +  \dots  + (-1)^{-i_m} \varphi^{- i_m} ) + \varphi^{ - j_1 } + \varphi^{-j_2 } + \dots + \varphi^{-j_k } <1\]
(\ref{integer}) implies that
\[- (  (-1)^{ - i_1 }\varphi^{- i_1}     +  \dots  + (-1)^{-i_m} \varphi^{- i_m} ) + \varphi^{ - j_1 } + \varphi^{-j_2 } + \dots + \varphi^{-j_k }\] is an integer. So
\[- (  (-1)^{ - i_1 }\varphi^{- i_1}     +  \dots  + (-1)^{-i_m} \varphi^{- i_m} ) + \varphi^{ - j_1 } + \varphi^{-j_2 } + \dots + \varphi^{-j_k } = 0.\]
By (\ref{integer}),
\[N = L_{i_1} + L_{i_2} + \dots + L_{i_m} + d_0 \]
This is equivalent to the second identity of Theorem 1.
\end{proof}

The proof of Theorem 2 uses a similar method with more cases.

\begin{proof}[Proof of Theorem 2.]

Let $N$ be a positive integer with representation $\beta(N)$ as in ($\ref{bn}$), so
\[\beta(N) = d_L d_{L-1} \dots d_1 d_0 . d_{-1} d_{-2} \dots d_{R+1} d_R\]
Let $i_1 > \dots > i_m > 0$ be the positive indices with $d_i = 1$. Let $-j_1 > \dots > -j_n$ be the negative indices with $d_{-j}=1$. Notice that $i_m$ is the first positive index such that $d_i=1$, that is, $i_m$ is the $i$ in the statement of Theorem 2.
So
\[N = \varphi^{i_1} + \dots + \varphi^{i_m} + d_0 + \varphi^{-j_1} + \dots + \varphi^{-j_n}\]
Taking the algebraic conjugate of both sides, we get
\begin{align}
    N &= {\bar \varphi}^{i_1} + \dots + {\bar \varphi}^{i_m} + d_0 + {\bar \varphi}^{-j_1} + \dots + {\bar \varphi}^{-j_n} \nonumber \\
    &= (-1)^{j_n}\varphi^{j_n} + \dots + (-1)^{j_1}\varphi^{j_1} + d_0 + (-1)^{i_m}\varphi^{-i_m} + \dots + (-1)^{i_1}\varphi^{-i_1}
\end{align}
Adding and subtracting powers of $\bar{\varphi}$,
\begin{align}
N & = (-1)^{j_n}(\varphi^{j_n} + {\bar \varphi}^{j_n}) + \dots + (-1)^{j_1}(\varphi^{j_1} + {\bar \varphi}^{j_1}) \nonumber \\
 & \; \; - ((-1)^{j_n}{\bar \varphi}^{j_n} + \dots + (-1)^{j_1}{\bar \varphi}^{j_1}) + (-1)^{i_m}\varphi^{i_m} + \dots + (-1)^{i_1}\varphi^{-i_1} + d_0 \nonumber \\
 &= (-1)^{j_n}(\varphi^{j_n} + {\bar \varphi}^{j_n}) + \dots + (-1)^{j_1}(\varphi^{j_1} + {\bar \varphi}^{j_1}) \nonumber \\
 & \; \; - ({\varphi}^{-j_n} + \dots + {\varphi}^{-j_1}) + (-1)^{i_m}\varphi^{-i_m} + \dots + (-1)^{i_1}\varphi^{-i_1} + d_0 \nonumber
\end{align}
so
\begin{align}
N &= (-1)^{j_n}L_{j_n} + \dots + (-1)^{j_1}L_{j_1} \nonumber \\
&\; \; - ({\varphi}^{-j_1} + \dots + {\varphi}^{-j_n})  \nonumber \\
&\; \; +(-1)^{i_1}\varphi^{-i_1} + \dots + (-1)^{i_m}\varphi^{-i_m} + d_0 \label{integer2}
\end{align}
Using Lemma 1, we have
\begin{equation} \label{inq2}
    -1 < - ({\varphi}^{-j_1} + \dots + {\varphi}^{-j_n}) < 0
\end{equation}
Using Lemma 1 and Lemma 2
\begin{eqnarray}
    -1 < (-1)^{ - i_m }\varphi^{- i_m}    +  \dots  + (-1)^{-i_1} \varphi^{- i_1} < 0  \; \; \; {\rm for} \; i_1 \; {\rm odd} \label{oddi} \\
    0 < (-1)^{ - i_m }\varphi^{- i_m}    +  \dots  + (-1)^{-i_1} \varphi^{- i_1} < 1  \; \; \; {\rm for} \; i_1 \; {\rm even} \label{eveni}
\end{eqnarray}
(\ref{integer2}) implies that
\[- ({\varphi}^{-j_1} + \dots + {\varphi}^{-j_n}) +(-1)^{i_1}\varphi^{-i_1} + \dots + (-1)^{i_m}\varphi^{-i_m}\] is an integer.
If $i_1$ is odd, then (\ref{inq2}) add (\ref{oddi}) implies that
\[-2 < - ({\varphi}^{-j_1} + \dots + {\varphi}^{-j_n}) +(-1)^{i_1}\varphi^{-i_1} + \dots + (-1)^{i_m}\varphi^{-i_m} < 0\]
so
\[- ({\varphi}^{-j_1} + \dots + {\varphi}^{-j_n}) +(-1)^{i_1}\varphi^{-i_1} + \dots + (-1)^{i_m}\varphi^{-i_m} = -1\]
\[N = (-1)^{j_n}L_{j_n} + \dots + (-1)^{j_1}L_{j_1} + d_0 -1\]
This proves Identity 1 of Theorem 2.
If $i_1$ is even, then (\ref{inq2}) add (\ref{eveni}) implies that
\[-1 < - ({\varphi}^{-j_1} + \dots + {\varphi}^{-j_n}) +(-1)^{i_1}\varphi^{-i_1} + \dots + (-1)^{i_m}\varphi^{-i_m} < 1\]
so
\[- ({\varphi}^{-j_1} + \dots + {\varphi}^{-j_n}) +(-1)^{i_1}\varphi^{-i_1} + \dots + (-1)^{i_m}\varphi^{-i_m} = 0\]
\[N = (-1)^{j_n}L_{j_n} + \dots + (-1)^{j_1}L_{j_1} + d_0\]
This proves Identity 2 of Theorem 2.
\end{proof}

\

\begin{corollary}
Let $N$ be a positive integer with representation $\beta(N)$ as
\[\beta(N) = d_L d_{L-1} \dots d_1 d_0 . d_{-1} d_{-2} \dots d_{R+1} d_R\]
then
\begin{align*}
2N &= d_L L_{L} + d_{L-1}L_{L-1} + \dots + d_1 L_{1} + 2d_0 \\
&\; \; + d_R (-1)^{-R}L_{-R} + \dots + d_{-1} (-1)^{-1}L_{1} .
\end{align*}
\end{corollary}

\begin{proof}
Summing the identities of $N$ in Theorem 1 and Theorem 2, we get
\[2N = d_L L_{L} + d_{L-1}L_{L-1} + \dots + d_1 L_{1} + 2d_0 \\
\; \; + d_R (-1)^{-R}L_{-R} + \dots + d_{-1} (-1)^{-1}L_{1} .\]
\end{proof}
We give a different proof using Proposition 1.
\begin{proof} [Proof 2.]
Let $i_1 > \dots > i_m > -j_1 > \dots > -j_s$ be the list of indices such that $d_i=1$ and $i_k \geq 0$ and $-j_k < 0$. So
\[N = \varphi^{i_1} + \dots + \varphi^{i_m} + \varphi^{-j_1} + \dots + \varphi^{-j_s} \]
By the equivalence between condition 2. and condition 3. in Proposition 1, we have
\[NL_n = L_{n+i_1} + \dots + L_{n+i_m} + L_{n-j_1} + \dots + L_{n-j_s}\] for all integers $n$. Taking $n=0$,
\[NL_0 = L_{i_1} + \dots + L_{i_m} + L_{-j_1} + \dots + L_{-j_s}\]
Since $L_0=2$, and $L_{-n} = (-1)^n L_n$ \cite{7},
\[2N = L_{i_1} + \dots + L_{i_m} + (-1)^{j_1} L_{j_1} + \dots + (-1)^{j_s} L_{j_s}\]
which is equivalent to the identity in Corollary 1.
\end{proof}

  Email Address: edwardz2024@student.cis.edu.hk

\end{document}